\documentclass{article}%
\usepackage{amssymb}
\usepackage{amsmath}
\usepackage{amsfonts}
\usepackage{geometry}
\usepackage{graphicx}%
\providecommand{\U}[1]{\protect\rule{.1in}{.1in}}
\newtheorem{theorem}{Theorem}[section]

\newtheorem{corollary}[theorem]{Corollary}

\newtheorem{lemma}[theorem]{Lemma}

\newtheorem{proposition}[theorem]{Proposition}
\newtheorem{remark}[theorem]{Remark}

\newenvironment{proof}[1][Proof]{\noindent\textbf{#1.} }{\ \rule{0.5em}{0.5em}}
\begin{document}

\title{On Integer Images of Max-plus Linear Mappings}
\author{Peter Butkovi\v{c}\thanks{E-mail: p.butkovic@bham.ac.uk}\\School of Mathematics, University of Birmingham\\Birmingham B15 2TT, United Kingdom}
\maketitle

\begin{abstract}
Let us extend the pair of operations $\left(  \oplus,\otimes\right)  =\left(
\max,+\right)  $ over real numbers to matrices in the same way as in
conventional linear algebra.

We study integer images of mappings $x\rightarrow A\otimes x$, where
$A\in\mathbb{R}^{m\times n}$ and $x\in\mathbb{R}^{n}.$ The question whether
$A\otimes x$ is an integer vector for at least one $x\in\mathbb{R}^{n}$ has
been studied for some time but polynomial solution methods seem to exist only
in special cases. In the terminology of combinatorial matrix theory this
question reads: is it possible to add constants to the columns of a given
matrix so that all row maxima are integer? This problem has been motivated by
attempts to solve a class of job-scheduling problems.

We present two polynomially solvable special cases aiming to move closer to a
polynomial solution method in the general case.



\textbf{Dedicated to Professor Karel Zimmermann.}

\end{abstract}

\section{\bigskip Introduction}

Since the 1960s max-algebra provides modelling and solution tools for a class
of problems in discrete mathematics and matrix algebra. The key feature is the
development of an analogue of linear algebra for the pair of operations
$\left(  \oplus,\otimes\right)  $ where%
\[
a\oplus b=\max(a,b)
\]
and
\[
a\otimes b=a+b
\]
for $a,b\in\overline{\mathbb{R}}\overset{def}{=}\mathbb{R}\cup\{-\infty\}. $
This pair is extended to matrices and vectors as in conventional linear
algebra. That is if $A=(a_{ij}),~B=(b_{ij})$ and $C=(c_{ij})$ are matrices of
compatible sizes with entries from $\overline{\mathbb{R}}$, we write
$C=A\oplus B$ if $c_{ij}=a_{ij}\oplus b_{ij}$ for all $i,j$ and $C=A\otimes B$
if
\[
c_{ij}=\bigoplus\limits_{k}a_{ik}\otimes b_{kj}=\max_{k}(a_{ik}+b_{kj})
\]
for all $i,j$. If $\alpha\in\overline{\mathbb{R}}$ then $\alpha\otimes
A=\left(  \alpha\otimes a_{ij}\right)  $. For simplicity we will use the
convention of not writing the symbol $\otimes.$ Thus in what follows the
symbol $\otimes$ will not be used (except when necessary for clarity), and
unless explicitly stated otherwise, all multiplications indicated are in max-algebra.

The interest in max-algebra (today also called tropical linear algebra)
\bigskip was originally motivated by the possibility of dealing with a class
of non-linear problems in pure and applied mathematics, operational research,
science and engineering as if they were linear due to the fact that $\left(
\overline{\mathbb{R}},\oplus,\otimes\right)  $ is a commutative and idempotent
semifield. Besides the main advantage of using linear rather than non-linear
techniques, max-algebra enables us to efficiently describe and deal with
complex sets \cite{PB BMIS}, reveal combinatorial aspects of problems
\cite{Linalgcomb} and view a class of problems in a new, unconventional way.
The first pioneering papers appeared in the 1960s \cite{CG first}, \cite{CG0}
and \cite{Vorobjov}, followed by substantial contributions in the 1970s and
1980s such as \cite{CG1}, \cite{gondran minoux}, \cite{GM dioids},
\cite{KZimm} and \cite{Cohen et al}. Since 1995 we have seen a remarkable
expansion of this research field following a number of findings and
applications in areas as diverse as algebraic geometry \cite{mikhalkin} and
\cite{sturmfels}, geometry \cite{joswig}, control theory and optimization
\cite{baccelli}, phylogenetic \cite{speyer sturmfels}, modelling of the
cellular protein production \cite{Brackley} and railway scheduling
\cite{heidergott et al}. A number of research monographs have been published
\cite{baccelli}, \cite{PB book}, \cite{heidergott et al} and \cite{McEneaney}.
A chapter on max-algebra appears in a handbook of linear algebra \cite{hogben}
and a chapter on idempotent semirings can be found in a monograph on semirings
\cite{golan}.

Max-algebra covers a range of linear-algebraic problems in the max-linear
setting, such as systems of linear equations and inequalities, linear
independence and rank, bases and dimension, polynomials, characteristic
polynomials, matrix scaling, matrix equations, matrix orbits and periodicity
of matrix powers \cite{baccelli}, \cite{PB book}, \cite{CG1}, \cite{BS} and
\cite{heidergott et al}. Among the most intensively studied questions was the
\textit{eigenproblem}, that is the question, for a given square matrix $A $ to
find all values of $\lambda$ and non-trivial vectors $x$ such that $Ax=\lambda
x.$ This and related questions such as $z$-matrix equations $Ax\oplus
b=\lambda x$ \cite{BSS} have been answered \cite{Robustness paper},
\cite{CG1}, \cite{GM dioids}, \cite{gaubert thesis}, \cite{bapat first} and
\cite{PB book} with numerically stable low-order polynomial algorithms. The
same applies to the \textit{subeigenproblem} that is the problem of finding
solutions to $Ax\leq\lambda x$ \cite{SSB} and the \textit{supereigenproblem}
that is solution to $Ax\geq\lambda x,$ \cite{super ev} and \cite{ss supereig}.
Max-linear and integer max-linear programs have also been studied
\cite{KZimm}, \cite{PB book}, \cite{MLP}, \cite{gks} and \cite{PB+Marie}.

\bigskip A specific area of interest is in solving the above mentioned
problems with integrality requirements. It seems in general there is no
polynomial solution method to find an integer eigenvector of a real matrix in
max-algebra or to decide that there is none. A closely related \cite{PB+Marie}
is the question whether the mapping $x\rightarrow Ax$ has an integer image,
that is whether $Ax$ is an integer vector for at least one $x\in\mathbb{R}%
^{n}.$ The motivation for the latter comes from operational problems such as
the following job-scheduling task \cite{CG1} and \cite{PB book}: Products
$P_{1},...,P_{m}$ are prepared using $n$ machines (processors), every machine
contributing to the completion of each product by producing a component. It is
assumed that each machine can work for all products simultaneously and that
all these actions on a machine start as soon as the machine starts to work.
Let $a_{ij}$ be the duration of the work of the $j^{th}$ machine needed to
complete the component for $P_{i}$ $(i=1,...,m;j=1,...,n).$ If this
interaction is not required for some $i$ and $j$ then $a_{ij}$ is set to
$-\infty.$ The matrix $A=\left(  a_{ij}\right)  $ is called the
\textit{production matrix}. Let us denote by $x_{j}$ the starting time of the
$j^{th}$ machine $(j=1,...,n)$. Then all components for $P_{i}$ $(i=1,...,m)$
will be ready at time
\[
\max(x_{1}+a_{i1},...,x_{n}+a_{in}).
\]
Hence if $b_{1},...,b_{m}$ are given completion times then the starting times
have to satisfy the system of equations:%

\[
\max(x_{1}+a_{i1},...,x_{n}+a_{in})=b_{i}\text{ \ for all }i=1,...,m.
\]
Using max-algebra this system can be written in a compact form as a system of
linear equations:%
\begin{equation}
Ax=b. \label{one-sided}%
\end{equation}
A system of the form (\ref{one-sided}) is called a \textit{one-sided system of
max-linear equations} (or briefly a \textit{one-sided max-linear system }or
just a\textit{\ max-linear system}). Such systems are easily solvable \cite{CG
first}, \cite{KZimm} and \cite{PB book}, see also Section \ref{Sec Prelim}.
However, sometimes the vector $b$ of completion times is not given explicitly,
instead it is only required that completions of individual products occur at
discrete time intervals, for instance at integer times. This motivates the
study of integer images of max-linear mappings to which this paper aims to
contribute. More precisely, we deal with the question: Given a real matrix
$A,$ find a real vector $x$ such that $Ax$ is integer or decide that none
exists. In the terminology of combinatorial matrix theory this question reads:
is it possible to add constants to the columns of a given matrix so that all
row maxima are integer? We will call this problem the \textit{Integer Image
Problem} (IIP). This problem has been studied for some time \cite{PB+Marie
IEV}, \cite{PB+Marie} and \cite{Marie CC}, yet it seems to be still open
whether it can be answered in polynomial time. In this paper we present two
polynomially solvable special cases aiming to suggest a direction in which an
efficient method could be found for general matrices in the future. We also
provide a brief summary of a selection of already achieved results.

\section{ Definitions, notation and previous results \label{Sec Prelim}}

Throughout the paper we denote $-\infty$ by $\varepsilon$ (the neutral element
with respect to $\oplus$) and for convenience we also denote by the same
symbol any vector, whose all components are $-\infty,$ or a matrix whose all
entries are $-\infty.$ A matrix or vector with all entries equal to $0$ will
also be denoted by $0.$ If $a\in\mathbb{R}$ then the symbol $a^{-1}$ stands
for $-a.$ Matrices and vectors whose all entries are real numbers are called
\textit{finite}. We assume everywhere that $m,n\geq1$ are integers and denote
$M=\left\{  1,...,m\right\}  $ and $N=\left\{  1,...,n\right\}  .$

It is easily proved that if $A,B,C$ and $D$ are matrices of compatible sizes
(including vectors considered as $m\times1$ matrices) then the usual laws of
associativity and distributivity hold and also isotonicity is satisfied:%
\begin{equation}
A\geq B\Longrightarrow AC\geq BC\text{ \ and }DA\geq DB. \label{Isotonicity}%
\end{equation}

\bigskip

A square matrix is called \textit{diagonal} if all its diagonal entries are
real numbers and off-diagonal entries are $\varepsilon.$ More precisely, if
$x=\left(  x_{1},...,x_{n}\right)  ^{T}\in\mathbb{R}^{n}$ then $diag\left(
x_{1},...,x_{n}\right)  $ is the $n\times n$ diagonal matrix%
\[
\left(
\begin{array}
[c]{cccc}%
x_{1} & \varepsilon & ... & \varepsilon\\
\varepsilon & x_{2} & ... & \varepsilon\\
\vdots & \vdots & \ddots & \vdots\\
\varepsilon & \varepsilon & ... & x_{n}%
\end{array}
\right)  .
\]
The matrix $diag\left(  0\right)  $ is called the \textit{unit matrix} and
denoted $I.$ Obviously, $AI=IA=A$ whenever $A$ and $I$ are of compatible
sizes. A matrix obtained from a diagonal matrix by permuting the rows and/or
columns is called a \textit{generalized permutation matrix}. It is known that
in max-algebra generalized permutation matrices are the only type of
invertible matrices \cite{CG1} and \cite{PB book}.

If $A$ is a square matrix then the iterated product $AA...A$ in which the
symbol $A$ appears $k$-times will be denoted by $A^{k}$. By definition
$A^{0}=I$.

Given $A=(a_{ij})\in\overline{\mathbb{R}}^{n\times n}$ the symbol $D_{A}$
stands for the weighted digraph $\left(  N,E,w\right)  $ (called
\textit{associated with} $A$) where $E=\left\{  \left(  i,j\right)
;a_{ij}>\varepsilon\right\}  $ and $w\left(  i,j\right)  =a_{ij}$ for all
$(i,j)\in E.$ The symbol $\lambda(A)$ denotes the \textit{maximum cycle mean}
of $A$, that is:
\begin{equation}
\lambda(A)=\max_{\sigma}\mu(\sigma,A), \label{mcm}%
\end{equation}
where the maximization is taken over all elementary cycles in $D_{A},$ and
\begin{equation}
\mu(\sigma,A)=\frac{w(\sigma,A)}{l\left(  \sigma\right)  } \label{mean}%
\end{equation}
denotes the \textit{mean} of a cycle $\sigma$. With the convention
$\max\emptyset=\varepsilon$ the value $\lambda\left(  A\right)  $ always
exists since the number of elementary cycles is finite. It can be computed in
$O\left(  n^{3}\right)  $ time \cite{karp}, see also \cite{PB book}. We say
that $A$ is \textit{definite} if $\lambda\left(  A\right)  =0$ and
\textit{strongly definite} if it is definite and all diagonal entries of $A$
are zero.

Given $A\in\overline{\mathbb{R}}^{n\times n}$ it is usual \cite{CG1},
\cite{baccelli}, \cite{heidergott et al} and \cite{PB book} in max-algebra to
define the infinite series
\begin{equation}
A^{\ast}=I\oplus A\oplus A^{2}\oplus A^{3}\oplus...\text{ .} \label{Def Delta}%
\end{equation}
The matrix $A^{\ast}$ is called the \textit{strong transitive closure} of $A,$
or the \textit{Kleene Star}.

It follows from the definitions that every entry of the matrix sequence
\[
\left\{  I\oplus A\oplus A^{2}\oplus...\oplus A^{k}\right\}  _{k=0}^{\infty}%
\]
is a nondecreasing sequence in $\overline{\mathbb{R}}$ and therefore either it
is convergent to a real number (if bounded) or its limit is $+\infty$. If
$\lambda(A)\leq0$ then%
\[
A^{\ast}=I\oplus A\oplus A^{2}\oplus...\oplus A^{k-1}%
\]
for every $k\geq n$ and can be found using the Floyd-Warshall algorithm in
$O\left(  n^{3}\right)  $ time \cite{PB book}.

\bigskip The matrix\ $\lambda^{-1}A$ for $\lambda\in\mathbb{R}$ will be
denoted by $A_{\lambda}$ and $(A_{\lambda})^{\ast}$ will be shortly written as
$A_{\lambda}^{\ast}.$

\bigskip The \textit{eigenvalue-eigenvector problem }(briefly
\textit{eigenproblem}) is the following:

\textit{Given }$A\in\overline{\mathbb{R}}^{n\times n}$\textit{, find all
}$\lambda\in\overline{\mathbb{R}}$ \textit{(eigenvalues) and\ }$x\in
$\textit{\ }$\overline{\mathbb{R}}^{n},x\neq\varepsilon$%
\textit{\ (eigenvectors) such that }%
\[
Ax=\lambda x.
\]

This problem has been studied since the work of R.A.Cuninghame-Green
\cite{CG0}. An $n\times n$ matrix has up to $n$ eigenvalues with
$\lambda\left(  A\right)  $ always being the largest eigenvalue (called
\textit{principal}). This finding was first presented by R.A.Cuninghame-Green
\cite{CG1} and M.Gondran and M.Minoux \cite{gondran minoux}, see also
N.N.Vorobyov \cite{Vorobjov}. The full spectrum was first described by
S.Gaubert \cite{gaubert thesis} and R.B.Bapat, D.Stanford and P. van den
Driessche \cite{bapat first}. The spectrum and bases of all eigenspaces can be
found in $O(n^{3})$ time \cite{Robustness paper} and \cite{PB book}.

The aim of this paper is to study the existence of integer images of
max-linear mappings and therefore we summarize here only the results on finite
solutions and for finite $A$. For $A\in\mathbb{R}^{n\times n}$ and $\lambda
\in\mathbb{R}$ we denote
\[
V(A,\lambda)=\left\{  x\in\mathbb{R}^{n}:Ax=\lambda x\right\}  .
\]
In this case there are no eigenvalues other than the principal and we can
easily describe all eigenvectors:

\begin{theorem}
\label{Th Ray finiteness crit}\cite{CG0}, \cite{CG1}, \cite{gondran minoux}
\bigskip If $A\in\mathbb{R}^{n\times n}$ then $\lambda\left(  A\right)  $ is
the unique eigenvalue of $A$ and all eigenvectors of $A$ are finite. If $A$ is
strongly definite and $\lambda=\lambda\left(  A\right)  $ then%
\[
V\left(  A,\lambda\right)  =\left\{  A_{\lambda}^{\ast}u:u\in\mathbb{R}%
^{n}\right\}  .
\]

\end{theorem}

\bigskip In what follows $V\left(  A\right)  $ will stand for $V\left(
A,\lambda\left(  A\right)  \right)  .$

As usual for any $a\in\mathbb{R}$ we denote the lower integer part, upper
integer part and fractional part of $a$ by $\left\lfloor a\right\rfloor ,$
$\left\lceil a\right\rceil $ and $fr\left(  a\right)  .$ Hence $fr\left(
a\right)  =a-\left\lfloor a\right\rfloor .$ For any matrix $A$ the symbol
$\left\lfloor A\right\rfloor $ stands for the matrix obtained by replacing
every entry of $A$ by its lower integer part, similarly $\left\lceil
A\right\rceil $ and $fr\left(  A\right)  .$ The same conventions apply to
vectors. The set of integer eigenvectors of $A\in\mathbb{R}^{n\times n}$ will
be denoted by $IV\left(  A\right)  $, that is%
\[
IV\left(  A\right)  =V\left(  A\right)  \cap\mathbb{Z}^{n}.
\]
Given an $A\in\mathbb{R}^{m\times n}$ we will use the following notation:%
\[
\operatorname{Im}\left(  A\right)  =\left\{  Ax:x\in\mathbb{R}^{n}\right\}
\]
and%
\[
\operatorname*{IIm}\left(  A\right)  =\operatorname{Im}\left(  A\right)
\cap\mathbb{Z}^{m}.
\]
We call $\operatorname{Im}\left(  A\right)  $ $\left[  \operatorname*{IIm}%
\left(  A\right)  \right]  $ the \textit{image set} \textit{of} $A$
[\textit{integer image set} \textit{of} $A$].

If we randomly generate two real numbers then their fractional parts are
different with probability 1. Being motivated by this we say that a real
vector $v$ is \textit{typical}\ if no two components of $v$ have the same
fractional part. On the other hand if every component of a vector $v$ has the
same fractional part then we say that $v$ is \textit{uniform}. If every column
of a real matrix $A$ is typical [uniform] then we say that $A$ is
\textit{column typical [column uniform]}.

\begin{remark}
\label{Rem II when col uniform}Observe that $IIm\left(  A\right)
\neq\emptyset$ if $A$ has at least one uniform column.
\end{remark}

As the next theorem shows strongly definite matrices are an important class
for which there is an easy solution to the integer eigenvalue problem.

\begin{theorem}
\cite{PB+Marie IEV} \bigskip Let $A\in\mathbb{R}^{n\times n}$ be strongly
definite. Then

\begin{enumerate}
\item $IV\left(  A\right)  \neq\emptyset$ if and only if $\lambda\left(
\left\lceil A\right\rceil \right)  =0.$

\item If $IV\left(  A\right)  \neq\emptyset$ then $IV\left(  A\right)
=\left\{  \left\lceil A\right\rceil ^{\ast}z:z\in\mathbb{Z}^{n}\right\}  .$
\end{enumerate}
\end{theorem}

\bigskip The\textit{\ max-algebraic permanent} of a matrix $A\in
\mathbb{R}^{n\times n}$ is an analogue of the conventional permanent:%
\[
per(A)=\sum\nolimits_{\pi\in P_{n}}^{\oplus}\prod\nolimits_{i\in N}^{\otimes
}a_{i,\pi(i)}%
\]
where $P_{n}$ is the set of all permutations of $N.$ In conventional notation
this reads:%
\[
per(A)=\max_{\pi\in P_{n}}\sum\nolimits_{i\in N}a_{i,\pi(i)}%
\]
which is the optimal value for the linear assignment problem for the matrix
$A$ \cite{Linalgcomb}, \cite{Burkard AP} and \cite{PB book}. Using the
notation%
\[
w\left(  \pi,A\right)  =\sum\nolimits_{i\in N}a_{i,\pi(i)}%
\]
for $\pi\in P_{n}$ we can then define the set of optimal permutations:%
\[
ap\left(  A\right)  =\left\{  \pi\in P_{n}:w\left(  \pi,A\right)
=per(A)\right\}  .
\]

Uniqueness of optimal permutations plays a significant role in max-algebra,
see for instance the question of regularity of matrices \cite{Linalgcomb}. It
is also important for integer images as shown in Theorem
\ref{Th Int im for CTL matrices} below. Note that it follows from the
definitions that $IIm\left(  A\right)  =IIm\left(  AQ\right)  $ for any
generalized permutation matrix $Q.$ It is known \cite{PB book} that for every
$A\in\mathbb{R}^{n\times n}$ with $\left\vert ap\left(  A\right)  \right\vert
=1$ there exists a unique generalized permutation matrix $Q$ such that $AQ$ is
strongly definite.

\begin{theorem}
\label{Th Int im for CTL matrices}\cite{PB+Marie IEV} \bigskip Let
$A\in\mathbb{R}^{n\times n}$ be column typical.

\begin{enumerate}
\item If $\left\vert ap\left(  A\right)  \right\vert >1$ then $IIm\left(
A\right)  =\emptyset.$

\item If $\left\vert ap\left(  A\right)  \right\vert =1$ and $Q$ is the
(unique) generalized permutation matrix such that $AQ$ is strongly definite
then%
\[
IIm\left(  A\right)  =IIm\left(  AQ\right)  =IV\left(  AQ\right)  .
\]

\end{enumerate}
\end{theorem}

\bigskip Theorem \ref{Th Int im for CTL matrices} effectively solves the IIP
for column typical square matrices. It is not difficult to see that in general
$m\leq n$ is a necessary condition for $\operatorname*{IIm}\left(  A\right)
\neq\emptyset$ if $A\in\mathbb{R}^{m\times n}$ is column typical. If $m\leq n$
then a necessary and sufficient condition for $\operatorname*{IIm}\left(
A\right)  \neq\emptyset$ is existence of a submatrix $A^{\prime}\in
\mathbb{R}^{m\times m}$ for which $\operatorname*{IIm}\left(  A^{\prime
}\right)  \neq\emptyset.$ Hence there is the possibility of solving the IIP
for $m\times n$ matrices by checking all $m\times m$ submatrices. The number
of such submatrices is $\binom{n}{m}$ which is polynomial when $m$ is fixed.
In particular, this immediately yields an $O\left(  n^{3}\right)  $ method for
answering the problem for $3\times n$ column typical matrices. One of the aims
of this paper is to present an $O\left(  n^{2}\right)  $ method for this
particular special case.

It will be useful to also define min-algebra over $\mathbb{R}$ \cite{CG1} and
\cite{PB book}:%
\[
a\oplus^{\prime}b=\min(a,b)
\]
and
\[
a\otimes^{\prime}b=a\otimes b
\]
for all $a$ and $b.$ We extend the pair of operations $\left(  \oplus^{\prime
},\otimes^{\prime}\right)  $ to matrices and vectors in the same way as in
max-algebra. We also define the \textit{conjugate} $A^{\#}=-A^{T}.$ Note that
isotonicity holds for $\left(  \oplus^{\prime},\otimes^{\prime}\right)  $
similarly as for $\left(  \oplus,\otimes\right)  $, see (\ref{Isotonicity}).

We will usually not write the operator $\otimes^{\prime}$and for matrices the
convention applies that if no multiplication operator appears then the product
is in min-algebra whenever it follows the symbol $\#$, otherwise it is in
max-algebra. In this way a residuated pair of operations (a special case of
Galois connection) has been defined, namely%
\begin{equation}
Ax\leq y\Longleftrightarrow x\leq A^{\#}y \label{residuation}%
\end{equation}
for all $x,y\in\mathbb{R}^{n}.$ Hence $Ax\leq y$ implies $A(A^{\#}y)\leq y$.
It follows immediately that a one-sided system $Ax=b$ has a solution if and
only if $A\left(  A^{\#}b\right)  =b$ (see Corollary \ref{Cor 2} below) and
using isotonicity then the system $Ax\leq b$ always has an infinite number of
solutions with $A^{\#}b$ being the greatest solution.

\section{\bigskip Finding an integer image for a $3\times n$ matrix}

We start with historically the first result in max-algebra. In what follows if
$A\in\mathbb{R}^{m\times n}$ and $j\in N$ then $A_{j}$ will denote the
$j^{th}$ column of $A.$

\textbf{Problem P1:} Given $A\in\mathbb{R}^{m\times n}$ and $b\in
\mathbb{R}^{m}$, find all $x=\left(  x_{1},...,x_{n}\right)  ^{T}\in
\mathbb{R}^{n}$ such that $Ax=b$ or decide that none exist.

For $A=\left(  a_{ij}\right)  \in\mathbb{R}^{m\times n}$ and $b=\left(
b_{1},...,b_{m}\right)  ^{T}\in\mathbb{R}^{m}$ define
\begin{equation}
\overline{x}=A^{\#}b, \label{xbar def}%
\end{equation}
that is
\[
\overline{x}_{j}=\min_{i\in M}\left(  b_{i}-a_{ij}\right)  =-\max_{i\in
M}\left(  a_{ij}-b_{i}\right)
\]
for all $j\in N$ and%
\[
M_{j}\left(  A,b\right)  =\left\{  i\in M:\overline{x}_{j}=b_{i}%
-a_{ij}\right\}  ,j\in N.
\]
The notation $M_{j}\left(  A,b\right)  $ will be shortened to $M_{j}$ if no
confusion can arise. The answer to P1 is summarized in the following statement.

\begin{proposition}
\label{Prop first}\cite{CG first}, \cite{CG1}, \cite{PB book} Let
$\overline{x}$ be as defined in (\ref{xbar def}) and $x\in\mathbb{R}^{n}.$ Then

\begin{enumerate}
\item[(a)] $A\overline{x}\leq b.$

\item[(b)] $Ax\leq b$ if and only if $x\leq\overline{x}.$

\item[(c)] $Ax=b$ if and only if $x\leq\overline{x}$ and
\[
\bigcup\nolimits_{x_{j}=\overline{x}_{j}}M_{j}=M.
\]

\end{enumerate}
\end{proposition}

\begin{corollary}
\label{Cor 1}\bigskip$A\overline{x}=b$ if and only if $\bigcup\nolimits_{j\in
N}M_{j}=M.$
\end{corollary}

\begin{corollary}
\label{Cor 2}$Ax=b$ has a solution if and only if $A\overline{x}=b.$
\end{corollary}

\textbf{Problem P2:} Given $A=\left(  a_{ij}\right)  \in\mathbb{R}^{m\times
n},$ $b=\left(  b_{1},...,b_{m}\right)  ^{T}\in\mathbb{R}^{m}$ and $d=\left(
d_{1},...,d_{n}\right)  ^{T}\in\mathbb{R}^{n}$, find an $x=\left(
x_{1},...,x_{n}\right)  ^{T}\in\mathbb{R}^{n}$ such that $Ax=b,x\leq d$ or
decide that none exists.

\begin{proposition}
\label{Prop second}$\left(  \exists x\in\mathbb{R}^{n}\right)  Ax=b,x\leq
d\iff Az=b,$ where $z=d\oplus^{\prime}\overline{x}$.
\end{proposition}

\begin{proof}
"If" is obvious since $z\leq d.$

Suppose now that $Ax=b,x\leq d$ for some $x\in\mathbb{R}^{n}.$ By Proposition
\ref{Prop first} (c) then $x\leq\overline{x}$ and $\bigcup\nolimits_{x_{j}%
=\overline{x}_{j}}M_{j}=M.$ It follows from the definition of $z$ that%
\[
z_{j}=\overline{x}_{j}\iff\overline{x}_{j}\leq d_{j}%
\]
and $z\leq\overline{x}.$ Hence if $x_{j}=\overline{x}_{j}$ then $\overline
{x}_{j}\leq d_{j}.$ Therefore $z_{j}=\overline{x}_{j}$ and thus%
\[
M=\bigcup\nolimits_{x_{j}=\overline{x}_{j}}M_{j}\subseteq\bigcup
\nolimits_{z_{j}=\overline{x}_{j}}M_{j}\subseteq M,
\]
from which the statement follows.
\end{proof}

\bigskip\textbf{Problem P3:} Given $A\in\mathbb{R}^{m\times n},$ find a point
in $\operatorname*{IIm}\left(  A\right)  $ or decide that there is none.

Solution to P3 is easy for $m=2$ as can be seen from the next few lines where
we give a full description of $\operatorname{Im}\left(  A\right)  $ and
$\operatorname*{IIm}\left(  A\right)  $. The primary objective of this paper
is to present a solution for $m=3$ provided that $A$ is column typical, which
is done later on.

\begin{proposition}
\label{Prop 2xn}\cite{PB+Marie IEV} Let $A\in\mathbb{R}^{2\times n}.$ Then%
\[
\operatorname{Im}\left(  A\right)  =\left\{  \left(  y_{1},y_{2}\right)
^{T}\in\mathbb{R}^{2}:y_{2}-y_{1}\in\left[  \underline{\alpha}\left(
A\right)  ,\overline{\alpha}\left(  A\right)  \right]  \right\}  ,
\]
where
\[
\underline{\alpha}\left(  A\right)  =\min_{j\in N}\left(  a_{2j}%
-a_{1j}\right)
\]
and
\[
\overline{\alpha}\left(  A\right)  =\max_{j\in N}\left(  a_{2j}-a_{1j}\right)
.
\]

\end{proposition}

\begin{proof}
\bigskip Let $y\in\mathbb{R}^{2}.$ Then by Corollaries \ref{Cor 1} and
\ref{Cor 2} $y\in\operatorname{Im}\left(  A\right)  $ if and only if%
\[
a_{1j}-y_{1}\geq a_{2j}-y_{2}%
\]
and%
\[
a_{2l}-y_{2}\geq a_{1l}-y_{1}%
\]
for some $j,l\in N.$

Equivalently,%
\[
y_{2}-y_{1}\geq a_{2j}-a_{1j}%
\]
and%
\[
y_{2}-y_{1}\leq a_{2l}-a_{1l}%
\]
for some $j,l\in N,$ from which the statement follows.
\end{proof}

\bigskip Note that we will write shortly $\underline{\alpha},\overline{\alpha
}$ instead of $\underline{\alpha}\left(  A\right)  ,\overline{\alpha}\left(
A\right)  $ if no confusion can arise.

\begin{corollary}
\label{Cor 3}Let $A\in\mathbb{R}^{2\times n}.$ Then%
\[
\operatorname*{IIm}\left(  A\right)  =\left\{  \left(  y_{1},y_{2}\right)
^{T}\in\mathbb{Z}^{2}:y_{2}-y_{1}\in\left[  \underline{\alpha},\overline
{\alpha}\right]  \right\}  .
\]

\end{corollary}

\begin{corollary}
\label{Cor 4}Let $A\in\mathbb{R}^{2\times n}.$ Then $\operatorname*{IIm}%
\left(  A\right)  \neq\emptyset$ if and only if $\left[  \underline{\alpha
},\overline{\alpha}\right]  \cap\mathbb{Z}\neq\emptyset.$
\end{corollary}

\bigskip\textbf{Problem P4:} Given $A\in\mathbb{R}^{2\times n}$ and $L=\left(
l_{1},l_{2}\right)  ^{T},U=\left(  u_{1},u_{2}\right)  ^{T}\in\mathbb{R}^{2},$
describe the set%
\[
S=\left\{  y\in\operatorname*{IIm}\left(  A\right)  :L\leq y\leq U\right\}  .
\]
By Corollary \ref{Cor 3} the set $S$ (if non-empty) consists of integer points
on adjacent parallel line segments. We may assume $L\in\mathbb{Z}^{2} $ (or
take $\left\lceil L\right\rceil $ if necessary). An answer to P4 is in the
next proposition which follows from Corollary \ref{Cor 3} immediately.

\begin{proposition}
\label{Prop Description of integer images >=L}The set $S=\left\{
y\in\operatorname*{IIm}\left(  A\right)  :L\leq y\leq U\right\}  $ consists of
integer points on line segments described by the following conditions:%
\begin{equation}
\left.
\begin{array}
[c]{c}%
y_{2}=y_{1}+\alpha,\alpha\in\left[  \underline{\alpha},\overline{\alpha
}\right]  \cap\mathbb{Z},\\
l_{1}\leq y_{1}\leq u_{1},\\
l_{2}\leq y_{2}\leq u_{2}.
\end{array}
\right\}  \label{Conditions describing II>=L}%
\end{equation}

\end{proposition}

\textbf{Problem P5:} Given $A\in\mathbb{R}^{2\times n}$ and $L\in
\mathbb{R}^{2},$ find an $x\in\mathbb{R}^{n}$ satisfying the following
conditions:%
\begin{equation}
\left.
\begin{array}
[c]{c}%
Ax\in\mathbb{Z}^{2}\\
x\leq0\\
Ax\geq L
\end{array}
\right\}  \label{Pr.5 conditions}%
\end{equation}
or decide that there is none.

\bigskip In order to solve P5 we first prove a few auxiliary statements.

A set $S\subseteq\mathbb{R}^{n}$ is called \textit{max-convex} if $\lambda
x\oplus\mu y\in S$ for any $x,y\in S$ and $\lambda,\mu\in\mathbb{R}$
satisfying $\lambda\oplus\mu=0.$

\begin{proposition}
\label{Prop S is max convex}Let $A\in\mathbb{R}^{m\times n}$ and
$d\in\mathbb{R}^{n}.$ Then the sets%
\[
S=\left\{  x\in\mathbb{R}^{n}:Ax=b\right\}
\]
and%
\[
S^{\prime}=\left\{  x\in\mathbb{R}^{n}:Ax=b,x\leq d\right\}
\]
are max-convex.
\end{proposition}

\begin{proof}
Since $\left(  \overline{\mathbb{R}},\oplus,\otimes\right)  $ is a semifield
the proof follows the lines of the proofs of the corresponding conventional
statements:%
\begin{align*}
A\left(  \lambda x\oplus\mu y\right)   &  =\lambda Ax\oplus\mu Ay\\
&  =\lambda b\oplus\mu b\\
&  =\left(  \lambda\oplus\mu\right)  b\\
&  =0b=b
\end{align*}
and
\begin{align*}
\lambda x\oplus\mu y  &  \leq\lambda d\oplus\mu d\\
&  =\left(  \lambda\oplus\mu\right)  d\\
&  =0d=d.
\end{align*}

\end{proof}

\begin{proposition}
\label{Prop Every value attained}Let $S\subseteq\mathbb{R}^{n}$ be a
max-convex set and $f\left(  x\right)  =c^{T}x,$ where $c\in\mathbb{R}^{n}.$
If $f\left(  x\right)  \leq f\left(  y\right)  $ for some $x,y\in S$ then for
every $f\in\left[  f\left(  x\right)  ,f\left(  y\right)  \right]  $ there
exists a $z\in S$ such that $f\left(  z\right)  =f.$
\end{proposition}

\begin{proof}
Denote $f^{\prime}=f\left(  x\right)  $, $f^{\prime\prime}=f\left(  y\right)
$ and suppose $f\in\left[  f^{\prime},f^{\prime\prime}\right]  .$ Let
$\lambda=0,\mu=f-f^{\prime\prime}\leq0.$ Then $\lambda\oplus\mu=0$ and thus by
Proposition \ref{Prop S is max convex} $z\in S$ where $z=\lambda x\oplus\mu
y.$ Also,
\begin{align*}
f\left(  z\right)   &  =c^{T}\left(  \lambda x\oplus\mu y\right) \\
&  =\lambda c^{T}x\oplus\mu c^{T}y\\
&  =\lambda f^{\prime}\oplus\mu f^{\prime\prime}\\
&  =f^{\prime}\oplus f=f.
\end{align*}

\end{proof}

\begin{proposition}
\label{Prop F is a closed interval}Let $a,c\in\mathbb{R}^{n},b\in
\mathbb{R},f\left(  x\right)  =c^{T}x$ and%
\[
S=\left\{  x\in\mathbb{R}^{n}:a^{T}x=b\right\}  .
\]
Then the set $F=\left\{  f\left(  x\right)  :x\in S\right\}  $ is a non-empty
closed interval.
\end{proposition}

\begin{proof}
Let us denote $a=\left(  a_{1},...,a_{n}\right)  ^{T}$ and $c=\left(
c_{1},...,c_{n}\right)  ^{T}.$ Let $A$ be the $1\times n$ matrix $A=\left(
a_{1},...,a_{n}\right)  $ and as before $\overline{x}=A^{\#}b=\left(
b-a_{1},...,b-a_{n}\right)  ^{T}$. Then clearly $\overline{x}\in S$ and
$x\leq\overline{x}$ for any $x\in S$ by Proposition \ref{Prop first}. It
follows by isotonicity that $f\left(  x\right)  \leq f\left(  \overline
{x}\right)  $ for every $x\in S$ and so $f\left(  \overline{x}\right)  $ is an
upper bound of $F$ attained on $S$.

On the other hand, define for $k=1,...,n:$%
\[
x^{\left(  k\right)  }=\left(  x_{1}^{\left(  k\right)  },...,x_{k-1}^{\left(
k\right)  },x_{k}^{\left(  k\right)  },x_{k+1}^{\left(  k\right)  }%
,...,x_{n}^{\left(  k\right)  }\right)  ^{T},
\]
where $x_{k}^{\left(  k\right)  }=\overline{x}_{k}=b-a_{k}$ and $x_{j}%
^{\left(  k\right)  }$ is any value not exceeding $c_{k}+\overline{x}%
_{k}-c_{j}$ for $j\neq k$ . Hence $f\left(  x^{\left(  k\right)  }\right)
=c_{k}+\overline{x}_{k}.$ For every $x\in S$ there exists a $k\in N$ such that
$x_{k}=\overline{x}_{k}$ and for this $k$ we have:
\[
f\left(  x\right)  \geq c_{k}+\overline{x}_{k}=f\left(  x^{\left(  k\right)
}\right)  \geq\min_{j\in N}f\left(  x^{\left(  j\right)  }\right)  .
\]
Since $\min_{j\in N}f\left(  x^{\left(  j\right)  }\right)  =f\left(
x^{\left(  j_{0}\right)  }\right)  $ for some $j_{0}\in N$ and $x^{\left(
j_{0}\right)  }\in S$ we have that $F$ has a lower bound and this bound is
attained on $S.$ The statement now follows from Propositions
\ref{Prop S is max convex} and \ref{Prop Every value attained}.
\end{proof}

\begin{corollary}
\bigskip\label{Cor Every value attained for x<=d}Let $a,c\in\mathbb{R}%
^{n},b\in\mathbb{R},f\left(  x\right)  =c^{T}x,d\in\mathbb{R}^{n}$ and%
\[
S^{\prime}=\left\{  x\in\mathbb{R}^{n}:a^{T}x=b,x\leq d\right\}  .
\]
If $S^{\prime}\neq\emptyset$ then the set $F=\left\{  f\left(  x\right)  :x\in
S^{\prime}\right\}  $ is a (non-empty) closed interval.
\end{corollary}

\begin{proof}
If $S^{\prime}\neq\emptyset$ then $\overline{x}_{j}\leq d_{j}$ for at least
one $j\in N.$ The rest of the proof follows the lines of the proof of
Proposition \ref{Prop F is a closed interval} where $\overline{x}$ is replaced
by $\overline{x}\oplus^{\prime}d.$ See Proposition \ref{Prop second}.
\end{proof}

\bigskip Let us return to P5. We denote by $X$ the set of vectors $x$
satisfying (\ref{Pr.5 conditions}). Note that we may assume \ without loss of
generality that $L=\left(  l_{1},l_{2}\right)  ^{T}\in\mathbb{Z}^{2}$
(otherwise we replace $L$ by $\left\lceil L\right\rceil $). By isotonicity we
have $Ax\leq A0$ for every $x\in X$ and we denote $A0$ by $U=\left(
u_{1},u_{2}\right)  ^{T}.$ If $L\leq U$ is not satisfied then $X=\emptyset$
hence we will assume in what follows that $L\leq U.$ So the task is to find
integer points $y=\left(  y_{1},y_{2}\right)  ^{T}$ in the rectangle $L\leq
y\leq U$ of the form $y=Ax,x\leq0$ or to decide that there are none. For ease
of reference we will denote%
\[
T=\left\{  y\in\mathbb{R}^{2}:L\leq y\leq U\right\}  .
\]
Recall that the integer points of the form $y=Ax$ in this rectangle are
described by (\ref{Conditions describing II>=L}). A little bit more
challenging is the task to identify those of them (if any) that are of the
form $y=Ax$ where $x\leq0.$

First we observe in the following statement that every integer point in $T$
(if any) can be "diagonally projected" on the left-hand side or bottom side of
$T.$

\begin{proposition}
\label{Prop Reflection of II in T}If $x\in X$ then there exists $\lambda\leq0$
such that the vector $x^{\prime}=\lambda x$ is in $X$ and satisfies either
$\left(  Ax^{\prime}\right)  _{1}=l_{1}$ or $\left(  Ax^{\prime}\right)
_{2}=l_{2}.$
\end{proposition}

\begin{proof}
Let us denote $Ax$ by $b=\left(  b_{1},b_{2}\right)  ^{T}.$

Suppose first that $b_{2}-b_{1}\geq l_{2}-l_{1}.$ Then take $\lambda
=l_{1}-b_{1}\leq0$ and observe%
\[
A\left(  \lambda x\right)  =\lambda\left(  Ax\right)  =\lambda b=\left(
l_{1},b_{2}+l_{1}-b_{1}\right)  \geq L.
\]
Clearly, $Ax^{\prime}=\lambda b\in\mathbb{Z}^{2}$ since both $\lambda
\in\mathbb{Z}$ and $b\in\mathbb{Z}^{2}.$ Also $x^{\prime}=\lambda x\leq0$,
thus $x^{\prime}\in X$ and $\left(  Ax^{\prime}\right)  _{1}=l_{1}$.

The case $b_{2}-b_{1}\leq l_{2}-l_{1}$ is proved similarly by taking
$\lambda=l_{2}-b_{2}\leq0.$
\end{proof}

\bigskip

We will also use the diagonal projection of the point $U$ (although it is not
integer) on the left-hand side or bottom side of $T$. For this we will need to
distinguish two possibilities to which we will refer as Case 1 and Case 2:

Case 1: $u_{2}-u_{1}\geq l_{2}-l_{1}.$

\bigskip Case 2: $u_{2}-u_{1}\leq l_{2}-l_{1}.$

Under the assumption of Case 1 the diagonal projection of $U$ is $P=\left(
p_{1},p_{2}\right)  ^{T}=\left(  l_{1},l_{1}+u_{2}-u_{1}\right)  ^{T}.$ In
Case 2 it is $P^{\prime}=\left(  l_{2}+u_{1}-u_{2},l_{2}\right)  ^{T}.$

Due to Proposition \ref{Prop Reflection of II in T} it is sufficient to search
integer points $Ax$ satisfying $\left(  Ax\right)  _{1}=l_{1}$ or $\left(
Ax\right)  _{2}=l_{2}$ and find those (if any) for which the condition
$x\leq0$ is satisfied. As there is possibly a non-polynomial number of such
points we will narrow the set of candidates. All candidates have the form
$\left(  l_{1},l_{1}+\alpha\right)  $ or $\left(  l_{2}-\alpha,l_{2}\right)
,$ where $\alpha\in\left[  \underline{\alpha}\left(  A\right)  ,\overline
{\alpha}\left(  A\right)  \right]  \cap\mathbb{Z}$ by Corollary \ref{Cor 3}.

Let us denote%
\begin{align*}
S &  =\left\{  Ax\in T:x\leq0\right\}  ,\\
S_{1} &  =\left\{  Ax\in T:x\leq0,\left(  Ax\right)  _{1}=l_{1}\right\}
\end{align*}
and%
\[
S_{2}=\left\{  Ax\in T:x\leq0,\left(  Ax\right)  _{2}=l_{2}\right\}  .
\]
Clearly, $U\in S$ and $P\in S_{1}$ in Case 1 and $P^{\prime}\in S_{2}$ in Case
2. We can also describe $S_{1}$ and $S_{2}$ as follows:%
\[
S_{1}=\left\{  \left(  l_{1},\left(  Ax\right)  _{2}\right)  ^{T}%
:x\leq0,\left(  Ax\right)  _{1}=l_{1}\right\}  \cap\left\{  \left(
l_{1},y_{2}\right)  :y_{2}\geq l_{2}\right\}
\]
and%
\[
S_{2}=\left\{  \left(  \left(  Ax\right)  _{1},l_{2}\right)  ^{T}%
:x\leq0,\left(  Ax\right)  _{2}=l_{2}\right\}  \cap\left\{  \left(
y_{1},l_{2}\right)  :y_{1}\geq l_{1}\right\}  .
\]
By Corollary \ref{Cor Every value attained for x<=d} the set $\left\{  \left(
l_{1},\left(  Ax\right)  _{2}\right)  ^{T}:x\leq0,\left(  Ax\right)
_{1}=l_{1}\right\}  $ is a closed interval or $\emptyset$ and so $S_{1}$ is a
closed interval or $\emptyset$, similarly $S_{2}.$ The point $P$ is in $S_{1}$
in Case 1 and $P^{\prime}$ in $S_{2}$ in Case 2, so at least one of the sets
$S_{1}$ and $S_{2}$ is non-empty in each case.

Consider now Case 1. Since the point $P=\left(  l_{1},l_{1}+u_{2}%
-u_{1}\right)  ^{T}$ is in $S_{1}\cup S_{2}$ it is easy to check whether a
point of the form $\left(  l_{1},l_{1}+\alpha\right)  $ or $\left(
l_{2}-\alpha,l_{2}\right)  $ is in $S_{1}\cup S_{2}$ where $\alpha\in\left[
\underline{\alpha}\left(  A\right)  ,\overline{\alpha}\left(  A\right)
\right]  \cap\mathbb{Z}$ - we only need to check those points of this form
that are closest to $P.$ More precisely, we distinguish 3 subcases:

\begin{itemize}
\item Subcase 1a: If $u_{2}-u_{1}<$ $\underline{\alpha}\left(  A\right)  $
then check the point $C_{1}=\left(  l_{1},l_{1}+\underline{\alpha}\left(
A\right)  \right)  ^{T}.$ Note that $\left(  l_{1},l_{1}+\underline{\alpha
}\left(  A\right)  \right)  ^{T}\geq L.$

\item Subcase 1b: If $u_{2}-u_{1}>$ $\overline{\alpha}\left(  A\right)  $ then
check the point $C_{2}=\left(  l_{1},l_{1}+\overline{\alpha}\left(  A\right)
\right)  ^{T}.$ If $l_{1}+\overline{\alpha}\left(  A\right)  <l_{2}$ then the
checkpoint is at the bottom side of $T,$ that is the point $\left(
l_{2}-\overline{\alpha}\left(  A\right)  ,l_{2}\right)  ^{T}\geq L.$

\item Subcase 1c: If $\underline{\alpha}\left(  A\right)  \leq u_{2}-u_{1}%
\leq$ $\overline{\alpha}\left(  A\right)  $ then check both
\[
C_{3}=\left(  l_{1},l_{1}+\left\lfloor u_{2}-u_{1}\right\rfloor \right)  ^{T}%
\]
and
\[
C_{4}=\left(  l_{1},l_{1}+\left\lceil u_{2}-u_{1}\right\rceil \right)  ^{T}.
\]
Note that both these points are $\geq L.$
\end{itemize}

Case 2 is treated similarly.

"Checking" a point $y$ means verifying that there is an $x\leq0$ for which
$Ax=y.$ By Proposition \ref{Prop second} this can be done by checking that
$A(\overline{x}\oplus^{\prime}0)=y$ which is $O\left(  n\right)  .$ Finding
$\underline{\alpha}\left(  A\right)  ,\overline{\alpha}\left(  A\right)  $ and
$U$ is obviously $O\left(  n\right)  $ as well so the whole method is
$O\left(  n\right)  .$

\textbf{Problem P6:} Given $A\in\mathbb{R}^{3\times n}$ find an $x\in
\mathbb{R}^{n}$ such that $Ax\in\operatorname*{IIm}\left(  A\right)  $ or
decide there is none.

\begin{remark}
\label{Rem 1}\bigskip Since $\left(  AD\right)  x=A\left(  Dx\right)  $ for
any $D=diag\left(  d_{1},...,d_{n}\right)  \in\overline{\mathbb{R}}^{n\times
n}$ we have $\operatorname*{IIm}\left(  AD\right)  =\operatorname*{IIm}\left(
A\right)  .$ Therefore in Problem 6 we may assume without loss of generality
that $a_{3j}=0$ for all $j\in N$ by taking $d_{j}=-a_{3j}$ for $j\in N$ if necessary.
\end{remark}

\begin{remark}
\label{Rem 2}Since $\left(  \alpha A\right)  x=A\left(  \alpha x\right)  $ for
any $\alpha\in\mathbb{R}$ we have $\operatorname*{IIm}\left(  \alpha A\right)
=\operatorname*{IIm}\left(  A\right)  .$ Therefore in Problem 6 we may assume
without loss of generality that for every $i\in\left\{  1,2,3\right\}  $ there
is an $x\in\mathbb{R}^{n}$ such that $\left(  Ax\right)  _{i}=0$ if
$\operatorname*{IIm}\left(  A\right)  \neq\emptyset$ by taking $\alpha
=-\left(  Ax\right)  _{i}$ if necessary.
\end{remark}

\begin{remark}
\label{Rem 3}Since $\alpha\left(  Ax\right)  =A\left(  \alpha x\right)  $ for
any $\alpha\in\mathbb{R}$ we have $A\left(  \alpha x\right)  \in\mathbb{Z}%
^{m}$ if $Ax\in\mathbb{Z}^{m}$ and $\alpha\in\mathbb{Z}.$ Therefore in Problem
6 we may assume without loss of generality that for every $j\in N$ there is an
$x\in\mathbb{R}^{n}$ satisfying $Ax\in\mathbb{Z}^{m}$ and $\left\lfloor
x_{j}\right\rfloor =0$ whenever $\operatorname*{IIm}\left(  A\right)
\neq\emptyset$ by taking $\alpha=-\left\lfloor x_{j}\right\rfloor $ if necessary.
\end{remark}

\bigskip We are now ready to present the main result of this paper - a
solution method for Problem 6 for column typical matrices. So let
$A\in\mathbb{R}^{3\times n}$ be a column typical matrix. We assume without
loss of generality (see Remark \ref{Rem 1}) that $a_{3j}=0$ for all $j\in N.$
Suppose that $Ax\in\mathbb{Z}^{3}$ for some $x\in\mathbb{R}^{n}$ and again
without loss of generality (see Remark \ref{Rem 2}) that $\left(  Ax\right)
_{3}=0.$ Hence there is a $k\in N$ such that $x_{k}=0\geq x_{j}$ for every
$j\in N.$ Let $\overline{A}$ be the matrix obtained from $A$ by removing row
3. Define%
\[
B^{\left(  k\right)  }=\left(  \overline{A}_{1},...,\overline{A}%
_{k-1},\overline{A}_{k+1},...,\overline{A}_{n}\right)
\]
and%
\[
z=\left(  x_{1},...,x_{k-1},x_{k+1},...,x_{n}\right)  ^{T}.
\]
Then%
\begin{equation}
\left.
\begin{array}
[c]{c}%
B^{\left(  k\right)  }z\in\mathbb{Z}^{2},\\
B^{\left(  k\right)  }z\geq\overline{A}_{k},\\
z\leq0.
\end{array}
\right\}  \label{Pr.6 conditions}%
\end{equation}
The first inequality follows from the assumption that $A$ is column typical
since then $a_{1k},a_{2k}$ are non-integer as they cannot have the same
fractional part as $a_{3k}$ which is zero (note that this inequality is
implicitly strict in each component).

Conversely, if a vector $z=\left(  z_{1},...,z_{n-1}\right)  ^{T}\in
\mathbb{R}^{n-1}$ satisfies (\ref{Pr.6 conditions}) then the vector
\[
x=\left(  z_{1},...,z_{k-1},0,z_{k},...,z_{n-1}\right)  ^{T}%
\]
satisfies $Ax\in\mathbb{Z}^{3}.$ So the method is to check for all $k=1,...,n$
that (\ref{Pr.6 conditions}) has a solution and find one. If every check fails
then $A$ has no integer image. Each check (for a fixed $k$) is an instance of
Problem 5 (with $A$ replaced by $B^{\left(  k\right)  }$ and $L$ is replaced
by $\overline{A}_{k}$), which can be solved in $O\left(  n\right)  $\ time, so
in total Problem 6 can be solved in $nO\left(  n\right)  =O\left(
n^{2}\right)  $ time.

We conjecture that the above mentioned method for solving Problem 6 in the
case $m=3$ and for column typical matrices can be extended to general matrices
and any $m.$ To do this one could try to develop a methodology to decide
whether a matrix $B$ obtained from $A$ with $\operatorname*{IIm}\left(
A\right)  \neq\emptyset$ by adding a row also has $\operatorname*{IIm}\left(
B\right)  \neq\emptyset.$ In the next section we will show that this can be
done if $A\in\mathbb{R}^{m\times n}$ is column uniform.

\section{Finding an integer image of an almost column uniform matrix}

\bigskip A matrix $A\in\mathbb{R}^{m\times n}$ is called \textit{almost column
uniform} if the matrix obtained by removing one row (called
\textit{exceptional}) of $A$ is column uniform.

\textbf{Problem P7:} Given an almost column uniform matrix $A\in
\mathbb{R}^{m\times n},$ find an $x\in\mathbb{R}^{n}$ such that $Ax\in
\mathbb{Z}^{m}$ or decide that there is none.

We will show in this section how to solve P7 in polynomial time.

We may assume without loss of generality that the following is satisfied:

\begin{enumerate}
\item The exceptional row is row $m,$ that is $A$ is of the form%
\[
\binom{\overline{A}}{a_{m1}...a_{mn}}%
\]
where $\overline{A}\in\mathbb{R}^{\left(  m-1\right)  \times n}$ is column uniform.

\item $A$ does not have a uniform column (see Remark
\ref{Rem II when col uniform}), that is the following hold for every $j\in N$
and for every $r,s=1,...,m-1$ :%
\[
fr\left(  a_{rj}\right)  =fr\left(  a_{sj}\right)  \neq fr\left(
a_{mj}\right)
\]

\item $a_{mj}=0$ for all $j\in N$ (see Remark \ref{Rem 1}). Observe that the
transformation suggested in Remark \ref{Rem 1} does not affect the assumption
that the matrix is almost column uniform.
\end{enumerate}

For every $j\in N$ the common value of $fr\left(  a_{ij}\right)  $ for all
$i=1,...,m-1$ will be denoted $f_{j}.$

\bigskip Suppose $Ax\in\mathbb{Z}^{m}$ for some $x\in\mathbb{R}^{n}.$ We may
assume without loss of generality (see Remark \ref{Rem 2}) that $\left(
Ax\right)  _{m}=0.$ Hence there is a $k\in N$ such that $x_{k}=0\geq x_{j}$
for every $j\in N.$ Let us now define%
\[
B^{\left(  k\right)  }=\left(  \overline{A}_{1},...\overline{A}_{k-1}%
,\overline{A}_{k+1},...,\overline{A}_{n}\right)
\]
and%
\[
z=\left(  x_{1},...,x_{k-1},x_{k+1,}...,x_{n}\right)  ^{T}.
\]
Since all entries of $\overline{A}_{k}$ are non-integral the vector $z$
satisfies%
\begin{equation}
\left.
\begin{array}
[c]{c}%
B^{\left(  k\right)  }z\in\mathbb{Z}^{m-1},\\
B^{\left(  k\right)  }z\geq\overline{A}_{k},\\
z\leq0.
\end{array}
\right\}  \label{Pr.7 conditions}%
\end{equation}
Conversely, if a vector $z=\left(  z_{1},...,z_{n-1}\right)  ^{T}\in
\mathbb{R}^{n-1}$ satisfies (\ref{Pr.7 conditions}) then $Ax\in\mathbb{Z}%
^{m},$ where%
\[
x=\left(  z_{1},...,z_{k-1},0,z_{k+1},...,z_{n-1}\right)  ^{T}.
\]

Let us denote $F^{\left(  k\right)  }=\left(  f_{1},...,f_{k-1},f_{k+1}%
,...,f_{n}\right)  ^{T}\in\mathbb{R}^{n-1}.$

\begin{proposition}
\label{Prop.7}The system (\ref{Pr.7 conditions}) has a solution if and only if
$-F^{\left(  k\right)  }$ is a solution.
\end{proposition}

We prove a few auxiliary statements before the proof of Proposition
\ref{Prop.7}.

\begin{lemma}
\label{L1}If $a,b\in\mathbb{R}$ then
\[
fr\left(  a+b\right)  =fr\left(  a\right)  +fr\left(  b\right)  \text{ if
}fr\left(  a\right)  +fr\left(  b\right)  <1
\]
and
\[
fr\left(  a+b\right)  =fr\left(  a\right)  +fr\left(  b\right)  -1\text{ if
}fr\left(  a\right)  +fr\left(  b\right)  \geq1.
\]

\end{lemma}

\begin{proof}
The statement follows immediately from the definition of the fractional part.
\end{proof}

\begin{lemma}
\label{L.2}If $a,b,y\in\mathbb{R},fr\left(  a\right)  =fr\left(  b\right)  $
then%
\[
\left\lfloor a+y\right\rfloor -a=\left\lfloor b+y\right\rfloor -b.
\]

\end{lemma}

\begin{proof}
For any $a,b,y\in\mathbb{R}$ we have%
\[
\left\lfloor a+y\right\rfloor -a=a+y-fr\left(  a+y\right)  -a=y-fr\left(
a+y\right)
\]
and similarly%
\[
\left\lfloor b+y\right\rfloor -b=y-fr\left(  b+y\right)  .
\]

If $fr\left(  a\right)  +fr\left(  y\right)  <1$ or, equivalently $fr\left(
b\right)  +fr\left(  y\right)  <1,$ then by Lemma \ref{L1} this implies%
\[
\left\lfloor a+y\right\rfloor -a=y-fr\left(  a\right)  -fr\left(  y\right)
=\left\lfloor y\right\rfloor -fr\left(  a\right)
\]
and similarly%
\[
\left\lfloor b+y\right\rfloor -b=\left\lfloor y\right\rfloor -fr\left(
b\right)  ,
\]
thus the statement follows.

The case $fr\left(  a\right)  +fr\left(  y\right)  \geq1$ is proved in the
same way.
\end{proof}

\begin{lemma}
\label{L3}Let $H=\left(  h_{ij}\right)  \in\mathbb{R}^{m\times n}$ be column
uniform, $w=\left(  w_{1},...,w_{n}\right)  ^{T}\in\mathbb{R}^{n}$ and
$\widetilde{w}=\left(  \widetilde{w}_{1},...,\widetilde{w}_{n}\right)  ^{T}$
be defined as follows:%
\[
\widetilde{w}_{j}=\left\lfloor h_{ij}+w_{j}\right\rfloor -h_{ij}%
\]
for any and therefore all $i\in M$ and for all $j\in N$ (see Lemma \ref{L.2}).
Then $\widetilde{w}\leq w$ and $H\widetilde{w}=\left\lfloor Hw\right\rfloor .$
\end{lemma}

\begin{proof}
For any $j\in N$ we have%
\[
\widetilde{w}_{j}\leq h_{ij}+w_{j}-h_{ij}=w_{j}.
\]
Let $i\in M.$ Then%
\[
\left\lfloor \max_{j}\left(  h_{ij}+w_{j}\right)  \right\rfloor =\max
_{j}\left\lfloor h_{ij}+w_{j}\right\rfloor =\max_{j}\left(  h_{ij}%
+\widetilde{w}_{j}\right)  .
\]

\end{proof}

\begin{proof}
[Proof of Proposition 4.1]Suppose that $z\in\mathbb{R}^{n-1}$ satisfies
(\ref{Pr.7 conditions}). Then
\[
B^{\left(  k\right)  }0\geq B^{\left(  k\right)  }z\geq\overline{A}_{k}%
\]
and $B^{\left(  k\right)  }z\in\mathbb{Z}^{m-1}.$ By taking $H=B^{\left(
k\right)  }$ and $w=0$ in Lemma \ref{L3} we get using the notation of Lemma
\ref{L3}:%
\begin{align*}
B^{\left(  k\right)  }\widetilde{w}  &  =B^{\left(  k\right)  }\widetilde{0}%
=\left\lfloor B^{\left(  k\right)  }0\right\rfloor \geq\left\lfloor B^{\left(
k\right)  }z\right\rfloor =B^{\left(  k\right)  }z\geq\overline{A}_{k}\\
B^{\left(  k\right)  }\widetilde{w}  &  \in\mathbb{Z}^{m-1}\\
\widetilde{w}  &  \leq0.
\end{align*}
Clearly, $\widetilde{0}=-F^{\left(  k\right)  },$ which completes the proof.
\end{proof}

Proposition \ref{Prop.7} provides a simple way of solving Problem 7. Since for
every $k\in N$ we have $-F^{\left(  k\right)  }\leq0$ and $B^{\left(
k\right)  }\left(  -F^{\left(  k\right)  }\right)  \in\mathbb{Z}^{m-1}$ we
only need to check whether for at least one $k\in N$ the vector $F^{\left(
k\right)  }$ satisfies
\begin{equation}
B^{\left(  k\right)  }\left(  -F^{\left(  k\right)  }\right)  \geq\overline
{A}_{k}. \label{Condition final}%
\end{equation}
If it does then $B^{\left(  k\right)  }\left(  -F^{\left(  k\right)  }\right)
$ extended by zero in the $m^{th}$ component is an integer image of $A$ and
$Ax\in\mathbb{Z}^{m}$ where%
\[
x=-\left(  f_{1},...,f_{k-1},0,f_{k+1},...,f_{n}\right)  ^{T}.
\]
If not then $\operatorname*{IIm}\left(  A\right)  =\emptyset.$

Checking the condition (\ref{Condition final}) is $O\left(  mn\right)  ,$
normalisation of $A$ with respect to the last row is also $O\left(  mn\right)
.$ In general this check needs to be done for all $k=1,...,n$ and so the
method is $O\left(  mn^{2}\right)  .$

\section{Conclusions}

We have shown in two special cases how to find an integer image of a matrix or
decide that none exists, where the matrix is obtained by adding one row to a
matrix for which the answer is known. It is self-suggesting an iterative
procedure for answering this question in a general case, however the
complexity issues remain to be solved.

\textbf{Acknowledgement}: This work was supported by the EPSRC grant EP/J00829X/1.

\end{document}